\documentclass{cimart}

\usepackage{pdfsync}
\usepackage{ytableau}
\usepackage[all]{xy}

\newcommand{\sudda}[1]{}

\let\<\langle
\let\>\rangle

\DeclareMathOperator{\adm}{adm}

\DeclareMathOperator{\Com}{Com}
\DeclareMathOperator{\Lie}{Lie}
\DeclareMathOperator{\M}{{M}}
\DeclareMathOperator{\Nov}{{Nov}}

\DeclareMathOperator{\wt}{{wt}}

\title{On the free metabelian Novikov and metabelian Lie-admissible algebras
    }

\author{% Please, use "Firstname Lastname" format, without abreviations
    Aigerim Dauletiyarova, Kanat Abdukhalikov and Bauyrzhan Sartayev
    }

\authorinfo[% Format: please, use "F. Name"
    Aigerim Dauletiyarova]{% Format: please, use "Department (only if 
    SDU University, Kaskelen, Kazakhstan}{%
    d\_aigera95@mail.ru
    }

\authorinfo[% Format: please, use "S. Name"
    Kanat Abdukhalikov]{% Format: please, use "Department (only if 
    UAE University, Al Ain, UAE}{%
    abdukhalik@uaeu.ac.ae
    }

\authorinfo[% Format: please, use "T. Name"
    Bauyrzhan Sartayev]{% Format: please, use "Department (only if 
    Narxoz University, Almaty, Kazakhstan and UAE University, Al Ain, UAE}{%
    baurjai@gmail.com
    }

\abstract{%
    In this paper, we consider Lie-admissible algebras, which are free Novikov and free Lie-admissible algebras with an additional metabelian identity. We construct a linear basis for both free metabelian Novikov and free metabelian Lie-admissible algebras. Additionally, we describe a space of symmetric polynomials for both the free metabelian Novikov algebra and the free metabelian Lie-admissible algebra.
    }

\keywords{% 2-5 keywords
    Novikov algebra, Lie-admissible algebra, metabelian identity, free algebra, polynomial identities.
    }

\msc{% Format: 1 code (primary); 0-4 codes (secondary).  See
     17A30 (primary), 17A50, 17D25 (secondary)
    }

\VOLUME{33}
\YEAR{2025}
\ISSUE{3}
\NUMBER{3}
\DOI{https://doi.org/10.46298/cm.12877}

\begin{document}

\section{Introduction} 
In recent years, algebras with metabelian identity have become popular objects in ring theory. Metabelian identity also can be stated as the solvability of index $2$. Various types of classical algebras with metabelian identity, such as Lie, Leibniz, Malcev, Jordan, etc. are considered. The variety of metabelian Lie algebras has attracted significant attention, see \cite{MS2022, Romankov}. A basis of the free metabelian Lie algebra was constructed in \cite{Bahturin1973}. In an analogical way, a basis of the free metabelian Leibniz algebra was constructed in \cite{MLeib}. Symmetric polynomials in the free metabelian Lie algebras were considered in \cite{DrenskyFindik}. The generators of symmetric polynomials in free metabelian Leibniz algebras were found in \cite{Findik}. Other examples of Lie-admissible algebras are assosymmetric algebras \cite{Zhakhayev}.
A basis of the free metabelian Malcev algebra is constructed in \cite{MMal't}. For this reason, we add metabelian identity to a well-known class of algebras which are Novikov and Lie-admissible.

Novikov algebras were introduced in a study of Hamiltonian operators concerning the integrability of certain partial differential equations~\cite{GelDor79}. Later, they played a significant role in a study of Poisson brackets of hydrodynamic type~\cite{BalNov85}.

It is well-known that, given a commutative algebra $A$ with a derivation $D$, the space $A$ under the product $x_1\circ x_2 = D(x_1)x_2$ forms a (right) Novikov algebra. Moreover, every Novikov algebra can be embedded into an appropriate commutative algebra with derivation $D$ \cite{BCZ2017}. Using rooted trees, the monomial basis of the free Novikov algebra in terms of $\circ$ was constructed in~\cite{DzhLofwall}. In terms of Young diagrams, the basis was developed in~\cite{DzhIsmailov}. By utilizing commutative algebra with derivation $D$ and a well-defined order, an alternative monomial basis of the free Novikov algebra is presented in \cite{GS2023}. The issues of solvability and nilpotency of Novikov algebras were addressed in \cite{SZ2020}. In section $2$, we construct a basis of the free solvable Novikov algebra with an index of $2$. In section $3$, we explicitly describe the symmetric polynomials for the multilinear part of the free metabelian Novikov algebra.

Let's shift our focus to metabelian Lie-admissible algebras. In the realm of Lie-admissible algebra theory, additional conditions like flexibility or power-associativity are crucial, leading to numerous noteworthy outcomes in this context, see \cite{3,18}. Another important direction of the research on Lie-admissible algebras concerns the property that their associator satisfies relations defined by a natural action of the symmetric group of degree 3 \cite{9,10}. A basis of the free Lie-admissible algebra and the Gr\"obner-Shirshov base theory for Lie-admissible algebras is given in \cite{CSZ2022}. In addition, there is given an analogue of PBW-theorem for the pair of Lie and Lie-admissible algebras. A basis of the free Lie-admissible algebra in terms of commutator and anti-commutator is given in \cite{Markl-Remm2004}. In Section~$4$, we construct the basis of a free metabelian Lie-admissible algebra in terms of commutator and anti-commutator. In Section~$5$, we explicitly describe the symmetric polynomials for the multilinear part of the free metabelian Lie-admissible algebra.

We consider all algebras over a field $\mathbb{K}$ of characteristic 0.

\section{Free metabelian Novikov algebra}

An algebra is called metabelian (right) Novikov if it satisfies the following identities:
\begin{equation}\label{lcom}
    a(bc)=b(ac),
\end{equation}
\begin{equation}\label{rsym}
    (ab)c-a(bc)=(ac)b-a(cb),
\end{equation}
\begin{equation}\label{metabelian}
    (ab)(cd)=0\end{equation}

Let $X=\{x_1,x_2\ldots\}$ be a countable set of generators. We denote by $\Nov\<X\>$ and $\M\Nov\<X\>$ the free Novikov and free metabelian Novikov algebra, respectively.

Let us denote by $\mathcal{N}_n$ the set of monomials of degree $n$ of the following form:
$$\mathcal{N}_n=\{(\cdots((x_{i_1}x_{i_2})x_{i_3})\cdots)x_{i_n}\; |\; i_2\leq i_3\leq \ldots\leq i_n\},$$
where $n\geq 5$.
We set
$$\mathcal{N}_1=\{x_i\}, \mathcal{N}_2=\{x_{i_1}x_{i_2}\}.$$
For degree $3$, $\mathcal{N}_3$ is defined as follows:
$$\mathcal{N}_3=\{(x_{i_1}x_{i_2})x_{i_3},x_{i_3}(x_{i_2}x_{i_1})\; |\; i_2\leq i_3\}.$$
For degree $4$, we define the set $\mathcal{N}_4$ as follows:
$$\mathcal{N}_4=\{ ((x_{i_1}x_{i_2})x_{i_3})x_{i_4}, x_{j_1}(x_{j_2} (x_{j_3}x_{j_4}))\;|\;i_2\leq i_3\leq i_4,\; j_1\leq j_2 \leq j_3\leq j_4\}.$$

\begin{lemma}\label{nilpidentities}
The algebra $\M\Nov\<X\>$ satisfies the following identities: 
\begin{multline}
  ((x_1(x_2x_3))x_4)x_5=(x_1((x_2x_3)x_4))x_5=(x_1(x_2(x_3x_4)))x_5=x_1(((x_2x_3)x_4)x_5)= \\
x_1((x_2(x_3x_4))x_5)=x_1(x_2((x_3x_4)x_5))=x_1(x_2(x_3(x_4x_5)))=0.
\end{multline}
\end{lemma}
\begin{proof}
Using (\ref{lcom}), (\ref{rsym}) and (\ref{metabelian}), one can obtain
\begin{multline*}
0=((x_1x_2)x_3)(x_4x_5)=x_4(((x_1x_2)x_3)x_5)=(x_1(x_2x_5))(x_4x_3)=x_4((x_1(x_2x_5))x_3)=\\
x_4(x_1((x_2x_5)x_3))+x_4((x_1x_3)(x_2x_5))
-x_4(x_1(x_3(x_2x_5)))=\\
x_4((x_2x_5)(x_1x_3))-x_4(x_1(x_3(x_2x_5)))=-x_4(x_1(x_3(x_2x_5))),
\end{multline*}
which gives
$$x_1(((x_2x_3)x_4)x_5)=x_1((x_2(x_3x_4))x_5)=x_1(x_2(x_3(x_4x_5)))=x_1(x_2((x_3x_4)x_5))=0.$$

By (\ref{lcom}), (\ref{rsym}) and (\ref{metabelian}), we have
\begin{multline*}
((x_1(x_2x_3))x_4)x_5=(x_1((x_2x_3)x_4))x_5+((x_1x_4)(x_2x_3))x_5-(x_1(x_4(x_2x_3)))x_5=\\
-x_1((x_4(x_2x_3))x_5)-(x_1x_5)(x_4(x_2x_3))+x_1(x_5(x_4(x_2x_3)))=x_1(x_5(x_4(x_2x_3)))=\\
-(x_4(x_2x_3))(x_1x_5)+x_1(x_5(x_4(x_2x_3)))=x_1(x_5(x_4(x_2x_3)))=0,
\end{multline*}
which gives
$$((x_1(x_2x_3))x_4)x_5=(x_1((x_2x_3)x_4))x_5=(x_1(x_2(x_3x_4)))x_5=0.$$
\end{proof}

\begin{theorem}\label{basisMLie}
The set $\bigcup \mathcal{N}_i$ is a linear basis of the algebra $\M\Nov\<X\>$.
\end{theorem}

\begin{proof}
The statement for degrees less than 5 can be verified by direct calculations. For $n\geq 5$, firstly, we show that every monomial of $\M\Nov\<X\>$ can be written as a sum of monomials from the set $\bigcup_{i\geq 5} \mathcal{N}_i$.
By Lemma \ref{nilpidentities}, we obtain that every monomial except $(\cdots((x_{i_1}x_{i_2})x_{i_3})\cdots)x_{i_n}$ is equal to $0$. By (\ref{rsym}) and Lemma \ref{nilpidentities}, we have
\begin{multline*}
    (\cdots((\cdots(x_{i_1}x_{i_2})\cdots)x_{i_{m-1}})x_{i_m})\cdots)x_{i_n}=
    (\cdots((\cdots(x_{i_1}x_{i_2})\cdots)(x_{i_{m-1}}x_{i_m}))\cdots)x_{i_n}+\\    (\cdots(((\cdots(x_{i_1}x_{i_2})\cdots)x_{i_{m}})x_{i_{m-1}})\cdots)x_{i_n}-(\cdots((\cdots(x_{i_1}x_{i_2})\cdots)(x_{i_{m}}x_{i_{m-1}}))\cdots)x_{i_n}=\\
(\cdots(((\cdots(x_{i_1}x_{i_2})\cdots)x_{i_{m}})x_{i_{m-1}})\cdots)x_{i_n},
\end{multline*}
i.e., the generators $x_{i_2}$, $x_{i_3},\ldots$, $x_{i_n}$ are rearrangeable. From these equations, we get that every monomial of $\M\Nov\<X\>$ which has a degree greater than 4 can be written as a sum of $\bigcup_{i\geq 5}\mathcal{N}_i$.

Now, we consider an algebra $A\<X\>$ with a basis monomials $\bigcup\mathcal{N}_i$ and multiplication $*$. Let us define a multiplication on monomials $\bigcup\mathcal{N}_i$ in $A\<X\>$ as follows:
\[
\begin{gathered}
\begin{cases}
    X_1*X_2=0\;\; \text{if $X_1,X_2\in\mathcal{N}_i$ and $\mathrm{deg}(X_1),\mathrm{deg}(X_2)>1$};\\
x_j * ((\cdots(x_{i_1}x_{i_2})\cdots) x_{i_n})=0; \\
((\cdots(x_{i_1}x_{i_2})\cdots) x_{i_n}) * x_j= ((\cdots(((\cdots(x_{i_1}x_{i_2})\cdots) x_{i_k}) x_j)x_{i_{k+1}}) \cdots) x_{i_n}, \\
\end{cases}
\end{gathered}
\]
where $n>4$ and  $i_2\leq \ldots\leq  i_k \leq j \leq i_{k+1} \leq \ldots \leq i_n$. Up to degree 4, we define multiplication in $A\<X\>$ that is consistent with identities (\ref{lcom}), (\ref{rsym}) and (\ref{metabelian}).
By straightforward calculation, we can check that an algebra $A\<X\>$ satisfies to (\ref{lcom}), (\ref{rsym}) and (\ref{metabelian}) identities. It remains to note that $A\<X\>\cong \M\Nov\<X\>$.
\end{proof}

\section{Symmetric polynomials of the free metabelian Novikov algebra}

Let $p(x_1,x_2,\ldots,x_n)$ be a polynomial of the free metabelian Novikov algebra generated by a finite set $X=\{x_1,x_2,\ldots,x_n\}$. The polynomial $p(x_1,x_2,\ldots,x_n)$ is called symmetric if it satisfies the following condition:
$$\sigma p(x_1,x_2,\ldots,x_n)=p(x_{\sigma(1)},x_{\sigma(2)},\ldots,x_{\sigma(n)})=p(x_1,x_2,\ldots,x_n),$$
where $\sigma\in S_n$. Let us define a set of polynomials $\mathcal{P}$ in $\M\Nov\<X\>$ as follows:
$$p_1=\sum_i x_i,\;p_2=\sum_{i\neq j} x_ix_j,$$
$$p_{3,1}=\sum_{i=1}^{n}\sum_{j_1< j_2}x_{j_2}(x_{j_1}x_i),\;p_{3,2}=\sum_{i=1}^{n}\sum_{j_1< j_2}(x_ix_{j_1})x_{j_2},$$
$$p_{4,1}=\sum_{j_1< j_2< j_3< j_4} x_{j_1}(x_{j_2}(x_{j_3}x_{j_4})),\;p_{4,2}=\sum_{i=1}^{n}\sum_{j_1< j_2< j_3} ((x_ix_{j_1})x_{j_2})x_{j_3},$$
and
$$p_n=\sum_{i=1}^{n}\sum_{j_1< j_2<\ldots< j_{n-1}} (\cdots((x_ix_{j_{1}})x_{j_{2}})\cdots)x_{j_{n-1}},$$
where $n\geq 5$.
The multilinear part of the free metabelian Novikov algebra is a space consisting of all elements containing each $x_i$ exactly once.
\begin{example}
For multilinear part of $\M\Nov\<X\>$, we obtain
$$p_{3,1}=x_2(x_1x_3)+x_3(x_2x_1)+x_3(x_1x_2),\;p_{3,2}=(x_1x_2)x_3+(x_2x_1)x_3+(x_3x_1)x_2,$$
$$p_{4,1}=x_1(x_2(x_3x_4)),\;p_{4,2}=((x_1x_2)x_3)x_4+((x_2x_1)x_3)x_4+((x_3x_1)x_2)x_4+((x_4x_1)x_2)x_3,$$
\begin{multline*}
p_n=(\cdots(((x_1x_2)x_3)x_4)\cdots)x_n+(\cdots(((x_2x_1)x_3)x_4)\cdots)x_n+\\(\cdots(((x_3x_1)x_2)x_4)\cdots)x_n+\ldots+(\cdots(((x_nx_1)x_2)x_3)\cdots)x_{n-1},    
\end{multline*}
where $n\geq 5$.
\end{example}

\begin{theorem}
For the multilinear part of the free metabelian Novikov algebra, the symmetric polynomials have the form $\mathcal{P}$.
\end{theorem}
\begin{proof}
For $n=1,2$, the result is obvious. For $n\geq 3$, we use the fact that every Novikov algebra is embeddable into commutative algebra with derivation and spanning elements of Novikov algebra in commutative algebra with derivation are monomials of weight $-1$ \cite{BCZ2017, KS}. The weight function $\wt(u)\in \mathbb Z$ is defined on monomials of commutative algebra with derivation $D$ by induction as follows,
\begin{gather*}
\wt(x)=-1,\quad x\in X; \\
\wt(d(u)) = \wt(u)+1; \quad \wt(uv)=\wt(u)+\wt(v).
\end{gather*}
For simplicity, we denote $D(x)$ and $D^n(x)$ by $x'$ and $(D^{n-1}(x))'$, respectively. For the multilinear part to describe the space of symmetric polynomials of $\Nov\<X\>$ in degree $3$, we need to find the space of the symmetric polynomials of the differential commutative algebra of weight $-1$ in degree $3$.
This is the monomials of the form
$$x_2'x_1'x_3+x_3'x_2'x_1+x_3'x_1'x_2\;\textrm{and}\;x_1''x_2x_3+x_2''x_1x_3+x_3''x_1x_2.$$
Rewriting the first polynomial by the operation of Novikov algebra, we obtain $p_{3,1}$. Rewriting the second polynomial, we obtain 
$$(x_1x_2)x_3-x_1(x_2x_3)+(x_2x_1)x_3-x_2(x_1x_3)+(x_3x_1)x_2-x_3(x_1x_2).$$
Adding to the last expression $p_{3,1}$, we obtain $p_{3,2}$. 

For degree $4$, we have $$\M\Nov\<X\>/\{(ab)(cd)\}\cong \Com\<X\>^{(D)}_{-1}/\{a''b'cd+a'b'c'd\},$$ where $\Com\<X\>^{(D)}_{-1}$ is a space of differential commutative algebra of weight $-1$. Hence, in degree $4$ of
$\Com\<X\>^{(D)}_{-1}/ \{a''b'cd+a'b'c'd\}$ we rewrite all monomials of the form $x_{i_1}''x_{i_2}'x_{i_3}x_{i_4}$ to $-x_{i_1}'x_{i_2}'x_{i_3}'x_{i_4}$. It remains to note that
$$0=a''b'cd+a'b'c'd-(a''b'dc+a'b'd'c)=a'b'c'd-a'b'd'c$$ which gives that we rewrite monomial $a'b'd'c$ to $a'b'c'd$.
Finally, the remained monomials of weight $-1$ are $x_{i_1}x_{i_2}x_{i_3}x_{i_4}'''$ and $x_{j_1}x_{j_2}'x_{j_3}'x_{j_4}'$, where $i_1\leq i_2\leq i_3$ and $j_1\leq j_2\leq j_3\leq j_4$. By Theorem \ref{basisMLie}, these monomials are linearly independent, and for multilinear part symmetric polynomials of $\Com\<X\>^{(D)}_{-1}/\{a''b'cd+a'b'c'd\}$ are
$$x_1'x_2'x_3'x_4\;\textrm{and}\;x_1'''x_2x_3x_4+x_2'''x_1x_3x_4+x_3'''x_1x_2x_4+x_4'''x_1x_2x_3,$$
which correspond to $p_{4,1}$ and $p_{4,2}$, analogically as in degree $3$.

By Theorem \ref{basisMLie}, starting from degree $5$, all monomials of $\Com\<X\>^{(D)}/\{a''b'cd+a'b'c'd\}$ of weight $-1$ except $x_i^{(n-1)}x_{j_{n-1}}\ldots x_{j_{1}}$ are equal to $0$, where $x_i^{(n-1)}=(x_i^{(n-2)})'$. Hence, starting from degree $5$, symmetric polynomials of $\Com\<X\>^{(D)}/$ $\{a''b'cd+a'b'c'd\}$ of weight $-1$ are
\begin{multline*}
x_1^{(n-1)}x_2x_3x_4\cdots x_n+x_2^{(n-1)}x_1x_3x_4\cdots x_n+\\
x_3^{(n-1)}x_1x_2x_4\cdots x_n+\ldots+x_n^{(n-1)}x_1x_2x_3\cdots x_{n-1},   
\end{multline*}
which correspond to $p_n$.
\end{proof}

\section{Free metabelian Lie-admissible algebra}

An algebra is called metabelian Lie-admissible if it satisfies the following identities:
\begin{multline}\label{lie-adm}
    (ab)c-(ba)c-c(ab)+c(ba)+(bc)a-(cb)a-a(bc)\\
+a(cb)+(ca)b-(ac)b-b(ca)+b(ac)=0,
\end{multline}
\begin{equation}\label{metabelianlie}
    (ab)(cd)=0.
\end{equation}

Let us consider the polarization of metabelian Lie-admissible algebra, i.e., we consider an algebra with two operations which is defined on metabelian Lie-admissible algebra as follows:
$$[a.b]=ab-ba,\;\{a,b\}=ab+ba.$$
In this case, the defining identities of the variety of metabelian Lie-admissible algebras become to
$$[a,b]=-[b,a],\;\{a,b\}=\{b,a\},$$
$$[[a,b],c]+[[b,c],a]+[[c,a],b]=0,$$
\begin{multline}\label{10}
[[a,b],[c,d]]=[[a,b],\{c,d\}]=[\{a,b\},\{c,d\}]=\{\{a,b\},\{c,d\}\}=\\
\{\{a,b\},[c,d]\}=\{[a,b],[c,d]\}=0.
\end{multline}
As a consequence, we obtain
\begin{equation}\label{11}
[[[a,b],c],d]=[[[a,b],d],c]    
\end{equation}
and
\begin{equation}\label{12}
[[\{a,b\},c],d]=[[\{a,b\},d],c]    
\end{equation}
which hold in free metabelian Lie-admissible algebra.

Let us construct a basis of the free metabelian Lie-admissible algebra in terms of binary trees with two types of vertices $\bullet$ and $\circ$. We consider only trees of the following form:

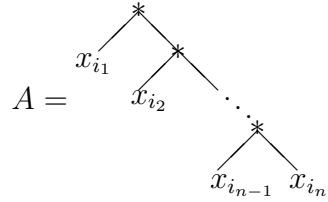
\begin{picture}(30,80)
\put(102,33){$A=$}
\put(147,68){$*$}
\put(150,70){\line(-1,-1){15}}
\put(150,70){\line(1,-1){15}}
\put(162,52){$*$}
\put(165,55){\line(-1,-1){15}}
\put(165,55){\line(1,-1){15}}
\put(182,33){\normalsize\rotatebox[origin=c]{320}{$\ldots$}}
\put(192,23){$*$}
\put(195,25){\line(-1,-1){15}}
\put(195,25){\line(1,-1){15}}
\put(126,49){$x_{i_1}$}
\put(147,35){$x_{i_2}$}
\put(177,3){$x_{i_{n-1}}$}
\put(207,3){$x_{i_{n}}$}
\end{picture}

We place on vertices of the tree $\bullet$ and $\circ$ in all possible ways. On the leaves of the tree, we place generators from the countable set $X$. This tree in a unique way corresponds to the sequence $(*_{x_{i_1}},*_{x_{i_2}},\ldots,*_{x_{i_{n-1}}},x_{i_{n}}).$

\begin{example}
If

\begin{picture}(30,80)
\put(100,60){$A=$}
\put(150,70){\line(-1,-1){15}}
\put(150,70){\line(1,-1){15}}
\put(165,55){\line(-1,-1){15}}
\put(165,55){\line(1,-1){15}}
\put(147,68){$\bullet$}
\put(162,52){$\circ$}
\put(126,49){$x_1$}
\put(147,35){$x_2$}
\put(177,35){$x_3$}
\end{picture}

then $A$ corresponds to $(\bullet_{x_1},\circ_{x_2},x_3)$. We define a set of sequences as follows:
\end{example}

1) If there are several consecutive black dots in a sequence, then all corresponding generators of these vertices are ordered, i.e., for
$(\ldots,\circ_{i_{k-1}},\bullet_{i_k},\bullet_{i_{k+1}},\ldots,\bullet_{i_{l-1}},\circ_{i_{l}},\ldots)$, we have $i_k\geq i_{k+1}\geq\ldots\geq i_{l-1}$;

2) If the rightmost vertex is white then the rightmost generator is less than the previous one, i.e., for $(\ldots,\circ_{i_{n-1}},x_{i_n})$, we have $i_{n-1}\leq i_n$;

3) If a given consecutive sequence of black dots continues to the rightmost vertex and the number of black dots is bigger than $2$, then all the generators of these vertices are ordered and the rightmost generator is bigger than the previous one, i.e., for
$(\ldots,\bullet_{i_k},\bullet_{i_{k+1}},\ldots,\bullet_{i_{n-1}},x_{i_n})$, we have $i_k\geq i_{k+1}\geq\ldots\geq i_{n-1}<i_n$;

4) In condition $3$ if the number of black dots is not bigger than $2$, then the generators are ordered as in Lyndon-Shirshov words, i.e., the basis monomials of the free Lie algebra of degrees 2 and 3;

For every such tree, we define a monomial from free metabelian Lie-admissible algebra as follows:
the tree with $n$ leaves is a right-normed monomial of degree $n$, i.e., this is the monomial $x_{i_1}*(x_{i_2}*(\ldots(x_{i_{n-1}}*x_{i_n})\ldots))$. The black multiplication $\bullet$ corresponds to the Lie bracket $[\cdot,\cdot]$ and white multiplication $\circ$ corresponds to $\{\cdot,\cdot\}$. We denote by $\mathcal{T}$ a set of trees that satisfy conditions (1), (2), (3) and (4), and we denote by $\mathcal{M}$ the set of monomials which correspond to the trees from $\mathcal{T}$.

\begin{theorem}
The set of monomials $\mathcal{M}$ is a basis of the free metabelian Lie-admissible algebra.
\end{theorem}
\begin{proof}
Firstly, we show that any monomial of the free Lie-admissible algebra can be written as a sum of monomials from $\mathcal{M}$.
After polarization of Lie-admissible algebra, one obtains that by commutative and anti-commutative identities on $\{\cdot,\cdot\}$ and $[\cdot,\cdot]$, respectively, and by (\ref{10}), any monomial can be written as a sum of right-normed monomials with multiplications $\{\cdot,\cdot\}$ and $[\cdot,\cdot]$. The condition (1) for right-normed monomials is provided by identities (\ref{11}) and (\ref{12}). The condition (3) is provided by the basis of the free metabelian Lie algebra, see \cite{Bahturin1973}. The conditions (2) and (4) are provided by commutativity of $\{\cdot,\cdot\}$ and identities of $[\cdot,\cdot]$, respectively.

Now, let us consider a free algebra $A\<X\>$ with the basis $\mathcal{M}$. The multiplications $\bullet$ and $\circ$ in $A\<X\>$ are defined as follows:
\[
\begin{gathered}
\begin{cases}
X_1*X_2=0\;\; \text{if $X_1,X_2\in\mathcal{M}$, $\mathrm{deg}(X_1),\mathrm{deg}(X_2)>1$ and $*$ is $\bullet$ or $\circ$};\\
(x_{i_1}*(\cdots*(x_{i_{n-1}}*x_{i_n})\cdots))\circ x_j=x_j \circ (x_{i_1}*(\cdots*(x_{i_{n-1}}*x_{i_n})\cdots)),\\
(x_{i_1}\circ(\cdots*(x_{i_{n-1}}*x_{i_n})\cdots))\bullet x_j=-x_j \bullet (x_{i_1}\circ(\cdots*(x_{i_{n-1}}*x_{i_n})\cdots)),\\
\text{where $n\geq 3$.}\\
(x_{i_1}\bullet(\cdots \bullet(x_{i_k}\bullet(x_{i_{k+1}}\circ(\cdots*(x_{i_{n-1}}*x_{i_n})\cdots)))\cdots))\bullet x_j=
-x_{i_1}\bullet(\cdots \bullet(x_{l}\\
\bullet(x_j\bullet(x_{l+1}\bullet(\cdots
\bullet(x_{i_k}\bullet(x_{i_{k+1}}\circ(\cdots *(x_{i_{n-1}}*x_{i_n})\cdots)))\cdots))))\cdots),\\
\text{where $l\leq j\leq l+1$ and $n\geq 3$.}\\
\end{cases}
\end{gathered}
\]
If the monomials $X_1$ and $X_2$ do not involve $\circ$ then we rewrite the product $X_1\bullet X_2$ according to the multiplication table of free metabelian Lie algebras. 

If the multiplications $\bullet$ and $\circ$ correspond to $[\cdot,\cdot]$ and $\{\cdot,\cdot\}$, respectively, then by straightforward calculations, we can check that an algebra $A\<X\>$ satisfies Jacobi identity and identities (\ref{10}), (\ref{11}), (\ref{12}). It remains to note that $A\<X\>\cong \M\Lie\textrm{-}\adm\<X\>$, where $\M\Lie\textrm{-}\adm\<X\>$ is a free metabelian Lie-admissible algebra.
\end{proof}

Calculating the dimension of operad $\M\Lie$ by means of the package \cite{DotsHij}, 
we get the following result:
\begin{center}
\begin{tabular}{c|ccccccc}
 $n$ & 1 & 2 & 3 & 4 & 5 & 6 & 7 \\
 \hline 
 $\dim(\M\Lie\textrm{-}\adm(n)) $ & 1 & 2 & 11 & 77 & 679 & 7184 & 88668
\end{tabular}
\end{center}
We see that the dimension of this operad is growing at a high rate, and this sequence does not coincide with any sequence from OEIS. In \cite{icecco} was given the dimension of the Lie-admissible operad up to degree $7$, which is
\begin{center}
\begin{tabular}{c|ccccccc}
 $n$ & 1 & 2 & 3 & 4 & 5 & 6 & 7 \\
 \hline
 $\dim(\Lie\textrm{-}\adm(n)) $ & 1 & 2 & 11 & 101 & 1299 & 21484 & 434314
\end{tabular}
\end{center}

It will be interesting to find a general formula for the dimension of the metabelian Lie-admissible and Lie-admissible operads.

\section{Symmetric polynomials of the free metabelian Lie-admissible algebra}

For $n$, let us define the sequence $(*_{x_{i_1}},*_{x_{i_2}},\ldots,*_{x_{i_{n-1}}},x_{i_{n}})$, where $*$ can be $\circ$ or $\bullet$. For each sequence, we define a space $(*_{x_{i_1}},*_{x_{i_2}},\ldots,*_{x_{i_{n-1}}},x_{i_{n}})_{(*,*,\ldots,*)}$ as follows:
$$(*_{x_{i_1}},*_{x_{i_2}},\ldots,*_{x_{i_{n-1}}},x_{i_{n}})_{(*,*,\ldots,*)}=\sum_{i_1,\ldots,i_n} x_{i_1}*(x_{i_2}*(\ldots (x_{i_{n-1}}*x_{i_{n}})\ldots)).$$

\begin{example}
$$(\bullet_{x_{i_1}},\circ_{x_{i_2}},x_{i_{3}})_{(\bullet,\circ)}=\sum_{i_1,i_2,i_3}x_{i_1}\bullet(x_{i_2}\circ x_{i_3}).$$
\end{example}

For each space $(*_{x_{i_1}},*_{x_{i_2}},\ldots,*_{x_{i_{n-1}}},x_{i_{n}})_{(*,*,\ldots,*)}$, $n\geq 3$, we define a polynomial $p_{(*,*,\ldots,*,n)}$ as follows:

1) We divide this sequence into consecutive vertices of the same colour;

2) For $k_i$ consecutive white vertices, we select $k_i$ generators and write all possible permutations for them;

3) For $k_i$ consecutive black vertices, we select $k_i$ generating ones and write them in descending order;

4) If the sequence ends with $k_i$ vertices of white colour, then for the selected generator we write all possible permutations of $k_i$ generators so that the last two generators are always ordered;

5) If the sequence ends with $k_i$ vertices of black colour, then for the selected generator we write symmetric polynomials of the free metabelian Lie algebra on $k_{i+1}$ generators.

%in descending order and the last two generators are always ordered;
%6) If the sequence ends with $k_i$ vertices of black colour and $k_i\leq 2$, 

The sum of such monomials gives polynomial $p_{(*,*,\ldots,*,n)}$. For $n=1,2$, we set
\[
    p_{(1)}=x_1+x_2+\ldots+x_n
    \quad \textup{ and } \quad 
    p_{(2)}=\sum_{i\neq j} x_i\circ x_j.
\]

\begin{example}
For multilinear part of $\M\Lie\textrm{-}\adm\<X\>$, let us construct $(\bullet_{x_{i_1}}, \circ_{x_{i_2}},$ $x_{i_3}{)}_{(\bullet,\circ)}$, ${(\circ_{x_{i_1}},\bullet_{x_{i_2}},x_{i_3})}_{(\bullet,\circ)}$ and ${(\circ_{x_{i_1}},\circ_{x_{i_2}},x_{i_3})}_{(\bullet,\circ)}$.
$$p_{(\bullet,\circ,3)}=x_1\bullet(x_2\circ x_3)+x_2\bullet(x_1\circ x_3)+x_3\bullet(x_1\circ x_2),$$
$$p_{(\circ,\bullet,3)}=x_1\circ(x_2\bullet x_3)+x_2\circ(x_1\bullet x_3)+x_3\circ(x_1\bullet x_2),$$
$$p_{(\circ,\circ,3)}=x_1\circ(x_2\circ x_3)+x_2\circ(x_1\circ x_3)+x_3\circ(x_1\circ x_2).$$
For
${(\bullet_{x_{i_1}},\bullet_{x_{i_2}},\circ_{x_{i_3}},x_{i_4})}_{(\bullet,\bullet,\circ)}$, we have
\begin{multline*}
p_{(\bullet,\bullet,\circ,3)}=x_2\bullet(x_1\bullet (x_3\circ x_4))+x_3\bullet(x_1\bullet (x_2\circ x_4))+x_4\bullet(x_1\bullet (x_2\circ x_3))+\\
x_3\bullet(x_2\bullet (x_1\circ x_4))+x_4\bullet(x_2\bullet (x_1\circ x_3))+x_4\bullet(x_3\bullet (x_1\circ x_2)).  
\end{multline*}
\end{example}

%Idea: $C_n^{k_1}\; C_{n-k_1}^{k_2}\; C_{n-k_1-k_2}^{k_3}\; \ldots \; C_{n-k_1-k_2-\ldots k_{m-1}}^{k_{m}}=C_n^{(k_1,k_2,\ldots,k_m)}.$    

For each $p_{(*,*,\ldots,*,n)}$, we replace $\bullet$ and $\circ$ to $[\cdot,\cdot]$ and $\{\cdot,\cdot\}$, respectively. Finally, we obtain the following result: 
\begin{theorem}
For the multilinear part of the free metabelian Lie-admissible algebra, the symmetric polynomials have the form $p_{(*,*,\ldots,*,n)}$.
%The space of symmetric polynomials of the free metabelian Lie-admissible algebra is generated by polynomials $p_{(*,*,\ldots,*,n)}$.
\end{theorem}
\begin{proof}
For $n=1,2$, the result is obvious. From the multiplication table of the free metabelian Lie-admissible algebra, one obtains
$$\M\Lie\textrm{-}\adm_{\geq 3}\<X\>=\oplus_{(*,*,\ldots,*)} (*_{x_{i_1}},*_{x_{i_2}},\ldots,*_{x_{i_{n-1}}},x_{i_{n}})_{(*,*,\ldots,*)}$$
as a vector space, where $\M\Lie\textrm{-}\adm_{\geq 3}\<X\>$ is a multilinear part of the free metabelian Lie-admissible algebra of degree greater than $2$. For example, if $n=3$ then
\begin{multline*}
\M\Lie\textrm{-}\adm_3\<X\>=(\bullet_{x_{i_1}},\bullet_{x_{i_2}},x_{i_3})_{(\bullet,\bullet)}\oplus(\bullet_{x_{i_1}},\circ_{x_{i_2}},x_{i_3})_{(\bullet,\circ)}\oplus\\
(\circ_{x_{i_1}},\bullet_{x_{i_2}},x_{i_3})_{(\circ,\bullet)}\oplus(\circ_{x_{i_1}},\circ_{x_{i_2}},x_{i_3})_{(\circ,\circ)}.
\end{multline*}
Moreover, each monomial of the space $(*_{x_{i_1}},\ldots,*_{x_{i_{n-1}}},x_{i_{n}})_{(*,*,\ldots,*)}$ under action $S_n$ belongs to the same space. It remains to note that $p_{(*,*,\ldots,*,n)}$ is a symmetric polynomial, and for each space $(*_{x_{i_1}}, *_{x_{i_2}}, \ldots, *_{x_{i_{n-1}}},$ $x_{i_{n}})_{(*,*,\ldots,*)}$ there is one unique symmetric polynomial which is $p_{(*,*,\ldots,*,n)}$.
\end{proof}

\subsection*{Acknowledgments}
This work was supported by UAEU grant G00004614. The authors are grateful to the anonymous referees for their valuable remarks that improved the text.

\EditInfo{January 17, 2024}{April 13, 2024}{David Towers and Ivan Kaygorodov}


\begin{thebibliography}{10}

\bibitem{Bahturin1973}
J.~A. Bahturin.
\newblock Identities in {L}ie algebras. I.
\newblock {\em Vestnik Moskov. Univ. Ser. I Mat. Meh.}, 28(1):12--18, 1973.\newblock [in Russian].

\bibitem{BalNov85}
A.~A. Balinskii and S.~P. Novikov.
\newblock Poisson brackets of hydrodynamic type, {F}robenius algebras and {L}ie algebras.
\newblock {\em Dokl. Akad. Nauk SSSR}, 283(5):1036--1039, 1985.\newblock [in Russian].

\bibitem{3}
G.~M. Benkart.
\newblock Power-associative {L}ie-admissible algebras.
\newblock {\em J. Algebra}, 90(1):37--58, 1984.

\bibitem{BCZ2017}
L.~A. Bokut, Y.~Chen, and Z.~Zhang.
\newblock Gr\"obner--{S}hirshov bases method for {G}elfand--{D}orfman--{N}ovikov algebras.
\newblock {\em J. Algebra Appl.}, 16(1):1750001, 2017.

\bibitem{CSZ2022}
Y.~Chen, I.~Shestakov, and Z.~Zhang.
\newblock Free {L}ie-admissible algebras and an analogue of the {PBW} theorem.
\newblock {\em J. Algebra}, 590:234--253, 2022.

\bibitem{DotsHij}
V.~Dotsenko and W.~Heijltjes.
\newblock Gr\"obner bases for operads, 2019.

\bibitem{MLeib}
V.~Drensky, G.~Piacentini Cattaneo, and G.~Maria.
\newblock Varieties of metabelian {L}eibniz algebras.
\newblock {\em J. Algebra Appl.}, 1(1):31--50, 2002.

\bibitem{DrenskyFindik}
V.~Drensky, S.~Fındık, and N.~S. Öuslu.
\newblock Symmetric polynomials in the free metabelian {L}ie algebras.
\newblock {\em Mediterr. J. Math.}, 17(5):11 pp., 2020.

\bibitem{DzhIsmailov}
A.~S. Dzhumadildaev and N.~A. Ismailov.
\newblock ${S}_n$- and ${GL}_n$-module structures on free {N}ovikov algebras.
\newblock {\em J. Algebra}, 416:287--313, 2014.

\bibitem{DzhLofwall}
A.~S. Dzhumadil'daev and C.~L\"ofwall.
\newblock Trees, free right-symmetric algebras, free {N}ovikov algebras and identities.
\newblock {\em Homology Homotopy Appl.}, 4(2):165--190, 2002.

\bibitem{Zhakhayev}
A.~S. Dzhumadil'daev, B.~K. Zhakhayev, and S.~A. Abdykassymova.
\newblock Assosymmetric operad.
\newblock {\em Commun. Math.}, 30(1):175--190, 2022.

\bibitem{GelDor79}
I.~M. Gelfand and I.~Y. Dorfman.
\newblock Hamilton operators and associated algebraic structures.
\newblock {\em Funct. Anal. Appl.}, 13(4):13--30, 1979.

\bibitem{9}
M.~Goze and E.~Remm.
\newblock Lie-admissible algebras and operads.
\newblock {\em J. Algebra}, 273(1):129--152, 2004.

\bibitem{10}
M.~Goze and E.~Remm.
\newblock A class of nonassociative algebras.
\newblock {\em Algebra Colloq.}, 14(2):313--326, 2007.

\bibitem{GS2023}
V.~Gubarev and B.~K. Sartayev.
\newblock Free special {G}elfand-{D}orfman algebra.
\newblock {\em J. Algebra Appl.}, 2023.

\bibitem{KS}
P.~S. Kolesnikov and B.~K. ~Sartayev.
\newblock On the special identities of {G}elfand-{D}orfman algebras.
\newblock {\em Exp. Math.}, 33(1), 165--174, 2024.

\bibitem{Markl-Remm2004}
M.~Markl and E.~Remm.
\newblock Algebras with one operation including {P}oisson and other {L}ie-admissible algebras.
\newblock {\em J. Algebra}, 299(1):171--189, 2006.

\bibitem{MS2022}
F.~A. Mashurov and B.~K. Sartayev.
\newblock Metabelian {L}ie and perm algebras.
\newblock {\em J. Algebra Appl.}, 23(4):2450065, 2024.

\bibitem{18}
H.~C. Myung.
\newblock Some classes of flexible {L}ie-admissible algebras.
\newblock {\em Trans. Amer. Math. Soc.}, 167:79--88, 1972.

\bibitem{MMal't}
S.~V. Pchelintsev.
\newblock Speciality of {M}etabelian {M}alcev {A}lgebras.
\newblock {\em Math. Notes}, 1-2:245–254, 2003.

\bibitem{Romankov}
V.~Romankov.
\newblock Generators of symmetric polynomials in free metabelian {L}eibniz algebras.
\newblock {\em Internat. J. Algebra Comput.}, 18(2):209--226, 2008.

\bibitem{icecco}
B.~Sartayev and A.~Ydyrys.
\newblock Free products of operads and {G}röbner base of some operads.
\newblock In {\em 2023 17th International Conference on Electronics Computer and Computation (ICECCO)}, pages 1--6, 2023.

\bibitem{SZ2020}
I.~Shestakov and Z.~Zhang.
\newblock Solvability and nilpotency of {N}ovikov algebras.
\newblock {\em Comm. Algebra}, 48(12):5412--5420, 2020.

\bibitem{Findik}
Z.~Özkurt and S.~Fındık.
\newblock Generators of symmetric polynomials in free metabelian {L}eibniz algebras.
\newblock {\em J. Algebra Appl.}, to appear.

\end{thebibliography}
\end{document}